\documentclass{amsart}[12pt, article]
\usepackage{amsmath}
\usepackage{mathrsfs}
\usepackage{amsfonts}
\usepackage{txfonts}
\usepackage{amssymb}

\newcommand\NN{{\mathbb N}}
\newcommand\RR{{{\mathbb R}}}

\def\SS {\mathbb{S}}

\newcommand\cA{{\mathcal A}}
\newcommand\cB{{\mathcal B}}
\newcommand\cD{{\mathcal D}}
\newcommand\cE{{\mathcal E}}

\newcommand\cL{{\mathcal L}}
\newcommand\cS{{\mathcal S}}
\newcommand\cT{{\mathcal T}}
\newcommand\cM{{\mathcal M}}
\newcommand\cN{{\mathcal N}}

\newcommand\pP{{\bf P}}
\newcommand\iI{{\bf I}}
\newcommand\p{{\partial}}

\newtheorem{theo}{Theorem}[section]
\newtheorem{lemm}[theo]{Lemma}

\newtheorem{prop}[theo]{Proposition}
\newtheorem{rema}[theo]{Remark}

\renewcommand{\theequation}{\thesection.\arabic{equation}}

\begin{document}

\title[Global existence for hard potential]
{ The Boltzmann equation without angular cutoff in the whole space: II, Global existence for hard potential}
\author{R. Alexandre}
\address{R. Alexandre,
IRENAV Research Institute, French Naval Academy
Brest-Lanv\'eoc 29290, France
\newline\indent
and \newline\indent
Department of Mathematics, Shanghai Jiao Tong University
\newline\indent
Shanghai, 200240, P. R. China}
\email{radjesvarane.alexandre@ecole-navale.fr}
\author{Y. Morimoto }
\address{Y. Morimoto, Graduate School of Human and Environmental Studies,
Kyoto University
\newline\indent
Kyoto, 606-8501, Japan} \email{morimoto@math.h.kyoto-u.ac.jp}
\author{S. Ukai}
\address{S. Ukai, 17-26 Iwasaki-cho, Hodogaya-ku, Yokohama 240-0015, Japan}
\email{ukai@kurims.kyoto-u.ac.jp}
\author{C.-J. Xu}
\address{C.-J. Xu, School of Mathematics and statistics, Wuhan University, 430072
Wuhan, P. R. China
\newline\indent
and \newline\indent
Universit\'e de Rouen, UMR 6085-CNRS,
Math\'ematiques
\newline\indent
Avenue de l'Universit\'e,\,\, BP.12, 76801 Saint
Etienne du Rouvray, France } \email{Chao-Jiang.Xu@univ-rouen.fr}
\author{T. Yang}
\address{T. Yang, Department of mathematics, City University of Hong Kong,
Hong Kong, P. R. China
\newline\indent
and \newline\indent
School of Mathematics, Wuhan University 430072,
Wuhan, P. R. China} \email{matyang@cityu.edu.hk}




\begin{abstract} As a continuation of our series  works on the  Boltzmann equation without angular cutoff assumption, in this part,
the global existence of solution to the Cauchy problem in the whole space
 is proved in some suitable weighted
Sobolev spaces for hard potential when the solution is a small perturbation of a global
equilibrium.

\end{abstract}

\maketitle

\section{Introduction}\label{s1}

This paper is
among the  series works on the Boltzmann equation with non-angular cutoff
cross-section and it follows
the  paper \cite{amuxy4-1} (herein referred as Part I), extending our initial work \cite{amuxy3,amuxy3b} on the same problem for
Maxwellian molecule. Consider
\begin{equation}\label{4-3-1.1}
f_t+v\cdot\nabla_x f=Q(f, f)\,,\,\,\,\,\,\,f|_{t=0}=f_0.
\end{equation}
Recall that the right hand side of (\ref{4-3-1.1}) is the
Boltzmann bilinear collision operator, which is given in the classical $\sigma -$representation by
\[
Q(g, f)=\int_{\RR^3}\int_{\mathbb S^{2}}B\left({v-v_*},\sigma
\right)
 \left\{g'_* f'-g_*f\right\}d\sigma dv_*\,,
\]
where $f'_*=f(t,x,v'_*), f'=f(t,x,v'), f_*=f(t,x,v_*), f=f(t,x,v)$, and for
$\sigma\in \mathbb S^{2}$,
$$
v'=\frac{v+v_*}{2}+\frac{|v-v_*|}{2}\sigma,\,\,\, v'_*
=\frac{v+v_*}{2}-\frac{|v-v_*|}{2}\sigma.
$$
As in our previous papers, we  assume
that the cross-section  takes  the form
\begin{equation}\label{4-3-1.2}
B(v-v_*, \cos \theta)=\Phi (|v-v_*|) b(\cos \theta),\,\,\,\,\,
\cos \theta=\frac{v-v_*}{|v-v_*|} \, \cdot\,\sigma\, , \,\,\,
0\leq\theta\leq\frac{\pi}{2},
\end{equation}
in which it contains a kinetic factor given by
$$
\Phi (|v-v_*|) =\Phi_\gamma(|v-v_*|)= |v-v_*|^{\gamma},
$$
and a factor related to the collision angle with singularity,
$$
\begin{array}{l}
 b(\cos \theta)\approx K\theta^{-2-2s} \
\
 \mbox{when} \ \ \theta\rightarrow 0+,
\end{array}
$$
for some constant $K>0$ and a parameter $0< s<1$. Notice that this includes
the potential of inverse power law as a special example.

And the setting of the problem is for perturbation of an equilibrium state, without loss of generality, that can be normalized as
\[
\mu (v)=(2\pi )^{-\frac{3}{2}}e^{-\frac{|v|^2}{2\,\,}}.
\]
 In order to avoid the unnecessary repetition,  readers can
refer to Part I,
comments and references. Here, we just refer the references \cite{Ce88,D-L,Grad,Pao,ukai-2,villani2} for the general background of the Boltzmann equation
and the recent progress on the mathematical theories for the case
without angular cutoff,
\cite{alex-review,al-1,amuxy-nonlinear-b,amuxy-nonlinear-3,amuxy3,amuxy3b,
amuxy4-1,amuxy4-4,gr-st-2,gr-st,mouhot-strain}. Hence, we now directly go to the Cauchy problem
for the perturbation denoted by
$g=\mu^{-\frac 12}(f-\mu)$
\begin{equation}\label{4-3-1.3}
\left\{\begin{array}{l} g_t + v\cdot\nabla_x g + \cL g=\Gamma(g,\,
g), \,\,\, t>0\,,\\
g|_{t=0}=g_0.\end{array}\right.
\end{equation}
In the following discussion, we will show that this equation can be
solved in some
weighted Sobolev spaces defined by: for
$k,\,\,\ell\in\RR$, set
$$
H^k_\ell(\RR^6_{x, v})=\left\{f\in\cS'(\RR^6_{x, v})\, ;\,\, W_\ell
f\in H^k(\RR^6_{x, v})\,\right\} ,
$$
where $\RR^6_{x, v}=\RR^3_{x}\times \RR^3_{v}$ and $ W_\ell
(v)=\langle v\rangle^\ell=(1+|v|^2)^{\ell/2}$ is the weight with respect to the
velocity variable $v\in\RR^3$.

Note that in Part I, we  introduced a new norm for the
description of the  dissipation  and coercivity of the
linearized collisional operator.
 Some properties of this norm together with some estimations on the
upper bounds for the nonlinear collision operator were also given
there. And we studied the Cauchy problem for the soft potential case, that is, the case when $\gamma +2s \le 0$ (recall that this terminology is an extension of the cutoff case, which loosely speaking corresponds to the case when $s=0$).

Along this direction, this paper is for  the hard potential case, that is,
 when $\gamma +2s >0$. Note that in particular  this   includes the
case of the Maxwellian molecule. But this latter case was already considered  in \cite{amuxy3, amuxy3b}.

The main result of this paper can be
 stated as follows.

\begin{theo}\label{4-3-theo1}
Assume that the cross-section satisfies \eqref{4-3-1.2} with
 $0<s<1$ and $ \gamma+2s> 0$.
Let $g_0 \in H^{k}_\ell(\RR^6)$ for some $k\geq 6,\, \ell> 3/2+2s+\gamma$. There exists $\varepsilon_0 >0$, such that if ~
$\|g_0\|_{H^{k}_\ell(\RR^6)}\leq \varepsilon_0$,  then the Cauchy problem
\eqref{4-3-1.3} admits a global  solution
$$
g\in
L^\infty([0, +\infty[\,;\,\,H^{k}_\ell(\RR^6)).
$$
\end{theo}

\begin{rema}
The uniqueness of the solution obtained in Theorem \ref{4-3-theo1} will be proved
in  \cite{amuxy4-4} in the general setting, together with the non-negativity of
$f=\mu+\mu^{\frac 12}g$. Recently, a similar global existence result for the inverse
power law was proved in \cite{gr-st-2, gr-st} by using different method in the setting of torus.
The method used here is simpler by using the newly introduced non-isotropic norm
which essentially captures the coercivity property of the linearized operator. Note that this
method can be applied to the Landau equation that leads to the same global existence result
obtained in \cite{guo}. Therefore, it is expected that this method can also be used for
other kinetic equations.
\end{rema}

The rest of the paper will be organized as follows.
In Section 2, we recall some basic lower and upper bound estimates on both the linearized and nonlinear operators from Parts I.
 With these estimates and some others
valid for hard potential,
 the local and global existences will be proved in Sections 3 and 4, respectively.

\renewcommand{\theequation}{\thesection.\arabic{equation}}
\section{Functional estimates of collision operators }\label{section2}
\smallskip
\setcounter{equation}{0}

First of all, let us recall the non-isotropic norm introduced in Part I
associated with $\cL$.  Corresponding to the cross-section $\Phi (|v-v_*|)b(\cos\theta)$, it is defined  by
\begin{align*}
||| g|||^2_{\Phi_\gamma}&= \iiint \Phi (|v-v_*|)b(\cos\theta) \mu_*\, \big(g'-g\,\big)^2\,\\
 &\,\,\,\,\,\,\,\,+\iiint \Phi (|v-v_*|)b(\cos\theta)g^2_* \big(\sqrt{\mu'}\,\, - \sqrt{\mu}\,\,
\big)^2\nonumber\,,
\end{align*}
where the integration  is over
~~$\RR^3_v\times\RR^3_{v_\ast}\times\SS^2_\sigma$. Without any ambiguity, sometimes we simply
use $\||\cdot\||$ for $\||\cdot\||_{\Phi_\gamma}$.

This norm was shown to be useful for the study on the soft potential. And
here, we will show that it works well for the study on the hard potential.
In the later discussion, we need the following
propositions proved in the previous two parts in this series.

\begin{prop}\label{prop2.1}(Prop. 2.1, \cite{amuxy4-1})  Assume that the cross-section satisfies \eqref{4-3-1.2} with
$0<s<1$ and $\gamma>-3$. Then for $g\in \mathcal{N}^{\perp}$
\begin{equation*}
||| g|||_{\Phi_\gamma}^2\lesssim \Big(\cL g,\, g\Big)_{L^2(\RR^3_v)}
\leq  2\Big(\cL_1 g,\, g\Big)_{L^2(\RR^3_v)}\lesssim ||| g|||_{\Phi_\gamma}^2,
\end{equation*}
where $\mathcal{N}$ is the null space of $\mathcal{L}$ defined
in Part I.
\end{prop}

\smallbreak
\begin{prop}\label{partI-prop2.2} (Prop. 2.2, \cite{amuxy4-1})
Assume that the cross-section satisfies \eqref{4-3-1.2}  with
$0<s<1$ and $\gamma>-3$.
Then
\[
\left\|
g\right\|^2_{H^s_{\frac{\gamma}{2}}(\RR^3_v)}+\left\|
g\right\|^2_{L^2_{s+\frac{\gamma}{2}}(\RR^3_v)}
\lesssim ||| g |||_{\Phi_\gamma}^2 \lesssim \left\|
g\right\|^2_{H^s_{s+\frac{\gamma}{2}}(\RR^3_v)}\,.
\]
\end{prop}

\begin{prop}\label{prop5.1} (Theorem 1.2, \cite{amuxy4-1})  Assume that  $0< s<1, \gamma+2s>0$. Then
\begin{align*}
\left| \big(\Gamma (f , g),\, h\big)_{L^2(\RR^3)} \right| &\lesssim \Big\{ \|f\|_{L^2_{s+\gamma/2}}|||g|||_{\Phi_\gamma}
+ \|g\|_{L^2_{s+\gamma/2}}|||f|||_{\Phi_\gamma} \Big\}|||h|||_{\Phi_\gamma} \, .\notag
\end{align*}
\end{prop}

Note that the above estimate on the nonlinear collision operator is not
enough for the proof of global existence because of the weight in $L^2_{s+\gamma/2}$. For this, we need to combine this with the following proposition.
For the proof of the following
proposition, we first recall an upper bound estimate for
a modified kernel
$\tilde\Phi_\gamma(z)=(1+|z|^2)^{\gamma/2}$, cf. Theorem 2.1 in  \cite{amuxy-nonlinear-3}. That is,
for any $0<s<1$, $\gamma \in\RR$ and any $m,
\,\alpha \in \RR$, we have
\begin{equation}\label{2.2+001}
\|Q_{\tilde\Phi_\gamma}(f,\, g)\|_{H^{m}_{\alpha}(\RR^3_v)}\lesssim \|f\|_{L^{1}_{{\alpha^+
+}( \gamma+2s)^+}(\RR^3_v)} \| g\|_{H^{m+2s}_{({\alpha +}
\gamma+2s)^+}(\RR^3_v)}\,.
\end{equation}

\begin{prop}\label{prop5.2} Let $0< s<1,  \gamma+2s>0$. Then
\begin{align*}
&\Big | \Big(\Gamma( f,g),\, h \Big ) \Big|
\lesssim |||h|||_{\Phi_\gamma}
\Big\{ ||f||_{L^{2}_{3/2+2s+\gamma+\epsilon}}\, || g
||_{H^{\max\{2s, 1\}}_{s+\gamma/2}}+|||f |||_{\Phi_\gamma} \|g\|_{L^\infty} \Big\}\,.
\end{align*}
\end{prop}

\begin{proof}
As in Part I, we apply the decomposition on the kinetic factor in the cross-section:
Let $0\le \phi(z)\le 1$ be a smooth radial function with $1$ for $z$ close to $0$,
and $0$ for large value of $z$. Set
\begin{equation*}
\Phi_\gamma (z) =\Phi_\gamma (z)\phi(z) +\Phi_\gamma (z)(1-\phi(z))=\Phi_c(z)+\Phi_{\bar{c}}(z).
\end{equation*}
We denote by $\Gamma_c(\cdot, \cdot), \Gamma_{\bar{c}}(\cdot, \cdot)$ the collision operators with the kinetic factors
in the cross-section given by $\Phi_c$ and $\Phi_{\bar{c}}$ respectively.
Note that
$$
\Big(\Gamma( f,  g),\, h \Big )=
\Big(\Gamma_{c}( f,  g),\, h \Big )+
\Big(\Gamma_{{\bar{c}}}( f,  g),\, h \Big ).
$$
Note that
\begin{align*}
&\Big(\Gamma_{\bar{c}}( f,\,
  g) ,\,\, h\Big)_{L^2(\RR^3_{v})}=\Big(Q_{{\bar{c}}}(\sqrt{\mu}\,
   f,\,  g) ,\,\,
 h\Big)_{L^2(\RR^3_{v})}\\
 &+ \iiint\Phi_{\bar{c}}(|v-v_*|) b(\cos\theta)
\Big(\sqrt{\mu_*}\, -\sqrt{\mu'_*}\,\,\Big) f'_*
g' h d v_* d \sigma dv\,.
\end{align*}
Since $\Phi_{\bar{c}}\lesssim\tilde\Phi_\gamma$, as shown in the Proposition 3.5 of
\cite{amuxy3}, we use (\ref{2.2+001}) with $m=0, \alpha=-s-\gamma/2$ to have
\begin{align*}
 \left|\Big(Q_{\bar{c}}(\sqrt{\mu}\,
  f,\,   g) ,\,\,
 h\Big)_{L^2(\RR^3_{v})}\right|&\lesssim \|h\|_{L^2_{s+\gamma/2}(\RR^3)}\| \sqrt{\mu}\,
  f\|_{L^1_{2s+\gamma}(\RR^3_v)}
 \|g\|_{H^{2s}_{s+\gamma/2}(\RR^3_v)}\\
& \lesssim \| f\|_{L^2(\RR^3)} ||g||_{H^{2s}_{s+\gamma/2}(\RR^3)} ||h||_{L^2_{s+\gamma/2}(\RR^3)}.
 \end{align*}
On the other hand, we can write
\begin{align*}
&\iiint \Phi_{\bar{c}}(|v-v_*|) b(\cos\theta) \Big(\sqrt{\mu_*}\,
-\sqrt{\mu'_*}\,\,\Big) f'_*
\,g' h d v_* d \sigma dv\,\\
&=\iiint \Phi_{\bar{c}}(|v-v_*|) b(\cos\theta) \Big(\sqrt{\mu_*}\,
-\sqrt{\mu'_*}\,\,\Big)f'_*
\,g' \Big(h-h'\Big) d v_* d \sigma dv\\
& +\iiint \Phi_{\bar{c}}(|v-v_*|) b(\cos\theta) \Big(\sqrt{\mu_*}\,
-\sqrt{\mu'_*}\,\,\Big) f'_*
\,g' h' d v_* d \sigma dv\\
&= D_1 + D_2\,.
\end{align*}
By the Cauchy-Schwarz inequality, one has
\begin{align*}
|D_1| &\leq \left(\iiint \Phi_{\bar c}(|v-v_*|) b(\cos\theta) |
f'_\ast |^2 |g '|^2 \Big((\mu_*)^{1/4} -
(\mu'_*)^{1/4}\Big)^2 d v_* d
\sigma dv\right)^{1/2}\\
&\qquad\times \left(\iiint\Phi_{\bar c}(|v-v_*|) b(\cos\theta) \Big( \mu^{1/4}_\ast +
(\mu'_*)^{1/4} \Big)^2 (h-h')^2 d v_* d \sigma dv\right)^{1/2}\, .
\end{align*}
As Lemma 2.6 in \cite{amuxy4-1}, for $\gamma+2s>0$, we have
\begin{align*}
&\iiint \Phi_{\bar c}(|v-v_*|) b(\cos\theta) |  f'_\ast |^2 | g
'|^2 \Big((\mu_*)^{1/4} - (\mu'_*)^{1/4}\Big)^2 d v_* d \sigma
dv\\
&\lesssim \iint |v-v_*|^{\gamma+2s} |  f_\ast |^2 | g|^2  d v_* dv\\
&\lesssim \iint |  f_\ast |^2 | g |^2 \langle v \rangle^{2s+\gamma} \,\langle v_* \rangle^{2s+\gamma} d v_* d dv\\
&\lesssim || f||^2_{L^2_{s+\gamma/2}(\RR^3)}
|| g ||^2_{L^2_{s+\gamma/2}(\RR^3)} \, ,
\end{align*}
and
\begin{align*}
& \iiint \Phi_{\bar c} b(\cos\theta) \Big( \mu^{1/4}_\ast + (\mu'_*)^{1/4}
\Big)^2 (h-h')^2 d v_* d \sigma dv\\
&\leq 4 \iiint \Phi_{\bar c} b(\cos\theta)  \mu^{1/2}_\ast (h-h')^2 d v_* d
\sigma dv\lesssim |||h|||^2_{\Phi_\gamma}.
\end{align*}
Therefore, we obtain
$$
|D_1|\lesssim ||f||_{L^2_{s+\gamma/2}(\RR^3)} ||g
||_{H^{2s}_{s+\gamma/2}(\RR^3)}|||h|||_{\Phi_\gamma}.
$$
For the term $D_2$, we have by using the symmetry in the integral to have
\begin{align*}
&\left|\iiint\Phi_{\bar c} b(\cos\theta) \Big(\sqrt{\mu_*}\,
-\sqrt{\mu'_*}\,\,\Big)f'_* \,g' h' d v_*
d \sigma dv\right|\\
&=\left|\iiint\Phi_{\bar c} b(\cos\theta) \Big(\sqrt{\mu'_*}\,
-\sqrt{\mu_*}\,\,\Big) f_* \,g h d v_* d
\sigma dv\right|\\
&\lesssim \int_{\RR^6_{v, v_*}}|
f_*|\, |g|\,\, |h| \langle
v\rangle^{2s+\gamma}\langle v_*\rangle^{2s+\gamma} d v_* d  dv\\
&\lesssim \| f\|_{L^1_{2s+\gamma}(\RR^3)}
\|\,g\|_{L^2_{s+\gamma/2}(\RR^3)} \|h\|_{L^2_{s+\gamma/2}(\RR^3)}\\
&\lesssim \| f\|_{L^2_{3/2+2s+\gamma+\epsilon}(\RR^3)} \|g\|_{L^2_{s+\gamma/2}(\RR^3)} \|h\|_{L^2_{s+\gamma/2}(\RR^3)}.
\end{align*}
Hence
$$
|D_2|\lesssim ||f||_{L^2_{3/2+2s+\gamma+\epsilon}(\RR^3)} \|g\|_{L^2_{s+\gamma/2}(\RR^3)} \|h\|_{L^2_{s+\gamma/2}(\RR^3)}.
$$
Therefore, it follows that
\[
\left|\Big(\Gamma_{\bar c}( f,\,g) ,\,\, h\Big)_{L^2(\RR^3)}\right|\lesssim
 ||f||_{L^2_{3/2+2s+\gamma+\epsilon}(\RR^3)}\, || g
||_{H^{2s}_{s+\gamma/2}}\,|||h|||_{\Phi_\gamma}.
\]
Next, similar to the arguments used
 in Part II, we have
\begin{align*}
\Big|\Big(\Gamma_c( f,\, g) ,\,\, ,h \Big ) \Big|
\lesssim&  \, \Big( \iiint  b  \Phi_c \mu_*^{1/2} \Big(f'_*
 g' - f_*\, g \Big) ^2  d \sigma dv dv_*   \Big)^{1/2} \, ||| h |||_{\Phi_c} \\
=&  {A}^{1/2} \, ||| h|||_{\Phi_c} \,.
\end{align*}
Since
$$
\Phi_c\leq \Phi_\gamma,
$$
we have $||| h|||_{\Phi_c}\leq ||| h|||_{\Phi_\gamma}$, and
$$
A 
\lesssim \Big\{|||f |||^2_{\Phi_\gamma} \|g\|^2_{L^\infty}
+ \|f\|^2_{L^2}\|g\|^2_{H^{\max(-\gamma/2, 1)}}
\Big \}\,.
$$
Then,  for $\gamma>-2s>-2$, we have $-\gamma/2<1$ and
\[
\left|\Big(\Gamma_c(f,\,
g) ,\,\, h\Big)_{L^2(\RR^3)}\right|\lesssim
 \Big\{|||f |||_{\Phi_\gamma} \|g\|_{L^\infty}
+ \|f\|_{L^2}\|g\|_{H^1}
\Big\}\,|||h|||_{\Phi_\gamma}.
\]
And this completes the proof of the proposition.
\end{proof}

In the following, we also need the following estimate on the commutator
of the weight function $W_\ell$ and the nonlinear collisional operator
$\Gamma(\,\cdot,\,\cdot\,)$ that follows from
Proposition 2.17 in \cite{amuxy4-1}.

\begin{prop}\label{weight-comm} Assume that  $0< s<1$ and
$\gamma+2s>0$. Then, for any $\ell\ge 0$, one has
\begin{equation*}
\Big| \big( W_\ell \Gamma (f, g) - \Gamma (f, W_\ell g),\,\,  h \big)_{L^2(\RR^3_v)}\Big|
\lesssim \Big\{ || f||_{L^2_{s+\gamma /2}} ||g ||_{L^2_{\ell+\gamma /2}} + || g||_{L^2_{s+\gamma /2}} ||f ||_{L^2_{\ell+\gamma /2}}\Big\} ||| h|||_{\Phi_\gamma}\, .
\end{equation*}
\end{prop}

\bigskip

\section{Local existence}\label{section4}
\smallskip \setcounter{equation}{0}

First of all,  the Leibniz formula gives
\begin{equation*}
\partial^\beta\Gamma(f,\, g)=
\sum_{\beta_1+\beta_2+\beta_3=\beta}
C_{{\beta_1},\beta_2, \beta_3}
\cT(\partial^{\beta_1} f,\,
\partial^{\beta_2} g,\, \mu_{\beta_3}\,)\, ,
\end{equation*}
with
\begin{equation*}
\cT(F,\, G,\, \mu_{\beta_3}\,)=Q(\mu_{\beta_3}\, F,\, G)+ \iint
\Phi(|v-v_*|)b(\cos\theta) \Big((\mu_{\beta_3})_*\, -(\mu_{\beta_3})'_*\,\,\Big)F'_* G'
d v_* d \sigma\,,
\end{equation*}
where $\mu_{\beta_3}=p_{\beta_3}(v)\sqrt{\mu(v)}=\partial^{\beta_3}
(\sqrt\mu\,)\,$ is a Maxwellian type function of the variable $v$ in the sense
that it is a product of a polynomial and a Gaussian.
As noted in the previous parts in this series,
 one can check that $\cT(F,\, G,\, \mu_{\beta_3}\,)$ enjoys the same properties as $\Gamma(F,\, G)$ stated above. Therefore, we will apply those estimates obtained
for $\Gamma(F,\, G)$ to $\cT(F,\, G,\, \mu_{\beta_3}\,)$.

Define the norm associated with the collision operator in
the variables $(x, v)$ by setting  for $m\in\NN, \ell\in \RR$,
$$
\cB^m_\ell(\RR^6_{x, v})=\left\{ g\in\cS'(\RR^6_{x, v});\,\,
||g||^2_{\cB^m_\ell(\RR^6)}=\sum_{|\beta|\leq m}\int_{\RR^3_x}|||W_\ell\,
\partial^\beta_{x, v} g(x,\, \cdot\,)|||^2_{\Phi_\gamma}dx <+\infty\right\}\,.
$$
First of all, recall

\begin{lemm}\label{lemm3.1.1} (Lemma 4.1, \cite{amuxy4-1})
 For any $\ell\geq 0, \alpha, \beta\in  \NN^3$,
\begin{equation*}
||W_\ell\partial^{\alpha}_x\partial^{\beta}_v\,\,\pP
g||_{\cB^0_0(\RR^6)}+||\pP (W_\ell
\partial^{\alpha}_x\partial^{\beta}_v\,\, g)||_{\cB^0_0(\RR^6)}
\leq C_{\ell, \beta}||\partial^{\alpha}_x
g||^2_{L^2(\RR^6)},
\end{equation*}
\begin{equation*}
C_0||g||^2_{\cB^0_0(\RR^6)}-C_1||g||^2_{L^2(\RR^6)} \leq \Big(\cL
g,\, g\Big)_{L^2(\RR^6)}
 \lesssim |||g|||^2_{\cB^0_0(\RR^6)},
\end{equation*}
where $C_0$ and $C_1$ are some positive constants, and
\begin{equation*}
||g||^2_{L^2_{l+s+\frac{\gamma}{2}}(\RR^6)} +||g||^2_{L^2(\RR^3_x; H^s_{l+\frac{\gamma}{2}}(\RR^3_v))} \lesssim ||g||^2_{\cB^0_l(\RR^6)}\lesssim ||g||^2_{L^2(\RR^3_x;
H^{s}_{l+s+\frac{\gamma}{2}}(\RR^3_v))}.
\end{equation*}
Here $\pP$ is the projection to the null space $\cN$.
\end{lemm}

We are now ready to prove the following estimate.

\begin{prop}\label{prop4.2}
Let $ \gamma+2s>0, N\geq 6, \ell> 3/2+2s+\gamma$. Then, for any $\beta\in\NN^6, |\beta|\leq N$,
\begin{equation*}
\left|\Big(W_\ell \partial^\beta_{x, v} \Gamma(f,\, g\,),\,\,
h\Big)_{L^2(\RR^6)}\right|\lesssim \Big\{ ||f||_{H^N_\ell(\RR^6)}\,\,
|| g||_{\cB^{N}_\ell(\RR^6)}+||f||_{\cB^N_\ell(\RR^6)}\,\,
|| g||_{H^{N}_\ell(\RR^6)}\Big\}\,\, || h||_{\cB^0_0(\RR^6)}.
\end{equation*}
\end{prop}
\begin{proof}
By using the  Leibniz formula, we have
\begin{align*}
&\Big(W_\ell \partial^\beta\Gamma(f,\, g),\, h\Big)_{L^2(\RR^6)}=
\sum_{\beta_1+\beta_2+\beta_3=\beta}
C_{\beta_1, \beta_2, \beta_3}\Big\{
\Big(\cT(\partial^{\beta_1} f,\,
W_\ell\,\partial^{\beta_2} g,\, \mu_{\beta_3}\,),\, h\Big)\,\\
&+\Big(W_\ell\,\cT(\partial^{\beta_1} f,\,
\partial^{\beta_2} g,\, \mu_{\beta_3}\,)-\cT(\partial^{\beta_1} f,\,
W_\ell\,\partial^{\beta_2} g,\, \mu_{\beta_3}\,),\, h\Big)
\Big\}\, .
\end{align*}
If $ |\beta_1|\leq N-3$, we get {}from Proposition \ref{prop5.1} that
\begin{align*}
& \left|\Big(\cT(\partial^{\beta_1} f,\, W_{\ell}\partial^{\beta_2}
g,\, \mu_{\beta_3}\,)
,\,\, h\Big)_{L^2(\RR^6)}\right|\\
&\lesssim \left(\int_{\RR^3_x}\Big( \|\partial^{\beta_1}
f\|^2_{L^2_{\gamma/2+s}(\RR^3_v)}|||W_\ell \partial^{\beta_2} g|||^2_{\Phi_\gamma}+\|\partial^{\beta_1} f||^2_{\Phi_\gamma}\|W_\ell\,\partial^{\beta_2} g||^2_{L^2_{s+\gamma/2}(\RR^3_v)}\Big)dx
\right)^{1/2}||h||_{\cB^0_0(\RR^6)}
\\
&\lesssim \Big(\|\partial^{\beta_1}f\|^2_{L^\infty(\RR^3_x;\,L^2_{\gamma/2+s}(\RR^3_v))}
||W_\ell\, \partial^{\beta_2} g||^2_{\cB^0_0(\RR^6)}
+||\partial^{\beta_1} f||^2_{L^\infty(\RR^3_x;\,H^s_{\gamma+2s}(\RR^3_v))}
||W_\ell \partial^{\beta_2} g||^2_{L^2_{s+\gamma/2}(\RR^6)}\Big)^{1/2}
||h||_{\cB^0_0(\RR^6)}
\\
&\lesssim \| f\|_{H^{N}_{\gamma+2s}(\RR^6)}
|||g|||_{\cB^{N}_\ell(\RR^6)} ||h||_{\cB^0_0(\RR^6)}.
\end{align*}
On the other hand, if $|\beta_1|>N-3$, then $|\beta_2|\leq 2\leq N-4$.
In this case,  Proposition \ref{prop5.2} implies,
\begin{align*}
& \left|\Big(\cT(\partial^{\beta_1} f,\, W_\ell\,\partial^{\beta_2}
g,\, \mu_{\beta_3}\,)
,\,\, h\Big)_{L^2(\RR^6)}\right|\\
&\lesssim ||h||_{\cB^0_0(\RR^6)} \Big( \|f\|_{H^{|\beta_1\,|}_{3/2+2s+\gamma+\epsilon}(\RR^6)} ||
W_\ell\,\partial^{\beta_2} g||_{L^\infty(\RR^3_x;\, H^{\max\{2s, 1\}}_{s+\gamma/2}(\RR^3_v))}
+ \|W_\ell\,\partial^{\beta_2} g\|_{L^\infty(\RR^6)} ||\partial^{\beta_1}
f||_{\cB^0_0}\Big)  \\
&\lesssim \Big(\| f\|_{H^{N}_{3/2+2s+\gamma+\epsilon}(\RR^6)} \,||W_\ell\, \partial^{\beta_2}
g||_{H^{\max\{2s, 1\}+3/2+\epsilon}_{s+\gamma/2}(\RR^6)}
+\| f\|_{\cB^{|\beta_1|}_0} \,||W_\ell\, \partial^{\beta_2}
g||_{H^{3+\epsilon}(\RR^6)}\Big)
 ||h||_{\cB^0_0(\RR^6)}\\
& \lesssim \Big(\| f\|_{H^{N}_\ell(\RR^6)} \,||
g||_{\cB^N_\ell}
+\| f\|_{\cB^{|\beta_1|}_0} \,||g||_{H^N_\ell(\RR^6)}\Big)
 ||h||_{\cB^0_0(\RR^6)}\, .
\end{align*}
Finally, Proposition \ref{weight-comm} yields
\begin{align*}
&\Big|\big(W_\ell\,\cT(\partial^{\beta_1} f,\,
\partial^{\beta_2} g,\, \mu_{\beta_3}\,)-
\cT(\partial^{\beta_1} f,\,
W_\ell\,\partial^{\beta_2} g,\, \mu_{\beta_3}\,),
\, h\big)_{L^2(\RR^6)}\Big|\\
& \lesssim \Big(\| f\|_{H^{N}_\ell(\RR^6)} \,||
g||_{\cB^N_\ell}
+\| f\|_{\cB^{N}_\ell} \,||g||_{H^N_\ell(\RR^6)}\Big)
 ||h||_{\cB^0_0(\RR^6)}\, .
\end{align*}
The combination of the above estimates completes
the proof of the proposition.
\end{proof}

For the linear operator $\cL_2$, Proposition 4.5 of \cite{amuxy4-1}
and the commutator estimate
give
\begin{prop}\label{prop4.4}
We have for any $\beta\in\NN^6$,
\begin{equation*}
\left|\Big(W_\ell \partial^\beta_{x, v} \cL_2(f),\,\,
h\Big)_{L^2(\RR^6)}\right|\leq C_{\ell,\, |\beta|} ||f||_{H^{|\beta\,|}_\ell(\RR^6)}\,\, ||\mu^{1/10^4}
h||_{L^2(\RR^6)}\, .
\end{equation*}
\end{prop}
By using the  interpolation inequalities
\begin{align*}
\|g\|_{H^{N}_{\ell+\gamma/2}}&\leq \epsilon \|g\|_{H^{N}_{\ell+\gamma/2+s}}
+C_\epsilon \|g\|_{H^{N}_\ell},
\end{align*}
for any small constant $\epsilon$, the following proposition follows from the same argument given in Proposition
4.8 of \cite{amuxy4-1}.

\begin{prop}\label{prop3.2.4}
Let $\gamma+2s>0, \beta\in\NN^6, |\beta|>0, \ell\geq 0$.
Then
\begin{align*}
&\left|\Big(\cL_1(W_\ell\,\partial^\beta_{x, v} g)-
\,W_\ell\, \partial^\beta_{x, v} \cL_1( g ),\,\,
h\Big)_{L^2(\RR^6)}\right|\\
&\qquad\lesssim \Big(||g||_{H^{|\beta|}_\ell(\RR^6)}+
||g||_{\cB^{|\beta|-1}_\ell(\RR^6)}
\Big)\,\, ||h||_{\cB^0_0(\RR^6)}\,\notag ,
\end{align*}
and for any $\epsilon>0$ there exists
a constant $C_\epsilon>0$ such that
\begin{align*}
&\left|\Big(\cL_1(W_\ell\, g)-
\,W_\ell\, \cL_1( g ),\,\,h\Big)_{L^2(\RR^6)}\right|\lesssim ||g||_{L^2_{\ell+\gamma/2}(\RR^6)}\,||h||_{\cB^0_0(\RR^6)}\\
&\qquad\leq\Big\{\epsilon||g||_{L^2_{\ell+\gamma/2+s}(\RR^6)} + C_\epsilon||g||_{L^2_{\ell}(\RR^6)}\Big\}\,
||h||_{\cB^0_0(\RR^6)}\notag\\
&\qquad\leq\Big\{\epsilon {||g||_{\cB^0_{\ell}(\RR^6)} } + C_\epsilon||g||_{L^2_{\ell}(\RR^6)}\Big\}\,
||h||_{\cB^0_0(\RR^6)}\,\notag .
\end{align*}

\end{prop}

We are now ready to show the local existence of solutions in
some weighted Sobolev spaces. Consider the following
Cauchy problem for a linear Boltzmann equation with a given
function $f$,
\begin{equation}\label{4.0.1}
\partial_t g + v\,\cdot\,\nabla_x g + \cL_1 g = \Gamma (f,\,g) -\cL_2 f
,\qquad g|_{t=0} = g_0\,,
\end{equation}
which is equivalent to the problem:
\begin{equation*}
\partial_t G + v\,\cdot\,\nabla_x G=
Q(F,\,G) ,\qquad G|_{t=0} = G_0,
\end{equation*}
with $F=\mu+\sqrt\mu\,f$ and $G=\mu+\sqrt\mu\,g$.

We shall now study the
energy estimates on (\ref{4.0.1}) in the function space
$H^N_\ell(\RR^6)$ for a smooth function $g$.
For $N\geq 6, \ell> 3/2+2s+\gamma$  and $\beta\in\NN^6,
|\beta|\leq N$, taking
$$
\varphi(t, x, v)= (-1)^{|\beta|}\partial^\beta_{x, v}
W_{2\ell}\partial^\beta_{x, v} g(t, x, v),
$$
as a test function on $\RR^3_x\times\RR^3_v$, we get
\begin{align*}
&\frac 12 \frac{d}{d t}\|\partial^\beta\,g\|^2_{L^2_\ell(\RR^6)}+
\Big(W_\ell\Big[\partial^\beta,\, v\,\Big]\,\cdot\,\nabla_x
g,\, W_\ell\partial^\beta g\Big)_{L^2(\RR^6)}+
\Big(W_\ell\partial^\beta \cL_1(g),\,
W_\ell\partial^\beta g\Big)_{L^2(\RR^6)}\\
&=\Big(W_\ell\partial^\beta \Gamma(f, \, g),\,
W_\ell\partial^\beta g\Big)_{L^2(\RR^6)}-
\Big(\partial^\beta \cL_2(f),\,
W_{2\ell}\partial^\beta g\Big)_{L^2(\RR^6)},
\end{align*}
where we have used the fact that
$$
\left( v\,\cdot\,\nabla_x \Big(W_\ell\partial^\beta\,
g\Big),\,W_\ell\partial^\beta g\right)_{L^2(\RR^6)}=0\, .
$$
Applying now Propositions \ref{prop4.2}, \ref{prop4.4}  and
\ref{prop3.2.4}, we get for $N\geq 6, \ell>3/2+2s+\gamma$ and $|\beta|\leq N$,
\begin{align*}
&\frac 12 \frac{d}{d t}\|\partial^\beta\,g\|^2_{L^2_\ell(\RR^6)}+
\Big(\cL_1\Big(W_\ell\,\partial^\beta_{x, v} g\Big),\,
W_\ell\,\partial^\beta_{x, v}g\Big)_{L^2(\RR^6)}\\
&\lesssim \Big\{ \|f\|_{H^N_\ell(\RR^6)} \,
||g||^2_{\cB^N_\ell(\RR^6)}+
\|f\|_{\cB^N_\ell(\RR^6)} \,\|g\|_{H^N_\ell(\RR^6)}
||g||_{\cB^{N}_\ell(\RR^6)}\\
&\qquad+ \|g\|^2_{H^{N}_\ell(\RR^6)}+
\|f\|^2_{H^N_\ell(\RR^6)} + \|g\|_{\cB^{N-1}_\ell(\RR^6)}
||g||_{\cB^N_\ell(\RR^6)}\Big\}.
\end{align*}
By induction on $\beta$ from
$|\beta|= 1$ to $N$,  the Cauchy-Schwarz inequality implies that
\begin{align*}
\frac{d}{d t}\|g\|^2_{H^N_\ell(\RR^6)}+\frac{C_0}{4}
||g||^2_{\cB^N_\ell(\RR^6)}&\lesssim \Big\{
\|f\|_{H^N_\ell(\RR^6)} \,
||g||^2_{\cB^N_\ell(\RR^6)}+
\|g\|^2_{H^N_\ell(\RR^6)} \,
||f||^2_{\cB^N_\ell(\RR^6)}\\
&\qquad+ \|g\|^2_{H^N_\ell(\RR^6)}+
\|f\|^2_{H^N_\ell(\RR^6)}+
\|g\|^2_{\cB^0_\ell(\RR^6)} \Big\}\, .
\end{align*}
On the other hand, taking $\beta=0$, we have
\begin{align*}
&\frac 12 \frac{d}{d t}\|g\|^2_{L^2_\ell(\RR^6)}+
\Big(\cL_1\Big(W_\ell\,g\Big),\,
W_\ell\,g\Big)_{L^2(\RR^6)}\\
&\lesssim \Big\{ \|f\|_{H^3_\ell(\RR^6)} \,
||g||^2_{\cB^0_\ell(\RR^6)}+
\|f\|_{\cB^3_\ell(\RR^6)} \,\|g\|_{L^2_\ell(\RR^6)}
||g||_{\cB^{0}_\ell(\RR^6)}\\
&\qquad+ \|g\|^2_{L^{2}_\ell(\RR^6)}+
\|f\|^2_{L^2_\ell(\RR^6)} +\epsilon\,
||g||^2_{\cB^0_\ell(\RR^6)}
\Big\}\,,
\end{align*}
which together with the coercivity estimate implies that
\begin{align*}
&\frac 12 \frac{d}{d t}\|g\|^2_{L^2_\ell(\RR^6)}+
\frac{C_0}{4}||g||^2_{\cB^0_\ell(\RR^6)}
\lesssim \Big\{ \|f\|_{H^3_\ell(\RR^6)} \,
||g||^2_{\cB^0_\ell(\RR^6)}\\
&\qquad\qquad+
\|f\|^2_{\cB^3_\ell(\RR^6)} \,\|g\|^2_{L^2_\ell(\RR^6)}
+ \|g\|^2_{L^{2}_\ell(\RR^6)}+
\|f\|^2_{L^2_\ell(\RR^6)}
\Big\}\, .
\end{align*}
In summary, we have shown that there exists a constant $\eta_0>0$ such that
for $N\geq 6, \ell>3/2+2s+\gamma$,
\begin{align*}
&\frac{d}{d t}\|g\|^2_{H^N_\ell(\RR^6)}+\eta_0
||g||^2_{\cB^N_\ell(\RR^6)}\lesssim \Big\{
\|f\|_{H^N_\ell(\RR^6)} \,
||g||^2_{\cB^N_\ell(\RR^6)}\\
&\qquad\qquad+
\|g\|^2_{H^N_\ell(\RR^6)} \,
||f||^2_{\cB^N_\ell(\RR^6)}+ \|g\|^2_{H^N_\ell(\RR^6)}+
\|f\|^2_{H^N_\ell(\RR^6)} \Big\}\, .\nonumber
\end{align*}

With the above differential inequality, the same argument for the
soft potential applies and it leads to the following theorem.

\begin{theo}\label{theo4.4}
Let $0<s<1$, $\gamma+2s>0, N\geq 6, \ell>3/2+2s+\gamma$. Assume that $g_0\in H^N_\ell(\RR^6)$ and
$f\in L^\infty([0, T];\,H^N_\ell(\RR^6))\bigcap L^2([0, T];\, \cB^N_\ell(\RR^6))$.  If $g$
$\in$ $ L^\infty([0,
T];\,H^N_\ell(\RR^6))\bigcap L^2([0, T];$ $\, \cB^N_\ell(\RR^6))$ is a
solution of the Cauchy problem (\ref{4.0.1}), then there exists
$\epsilon_0>0$ such that if
\begin{equation*}
\|f\|^2_{L^\infty([0, T];\,H^N_\ell(\RR^6))}+\|f\|^2_{L^2([0, T];\,\cB^N_\ell(\RR^6))}\leq \epsilon^2_0,
\end{equation*}
we have
\begin{equation*}
\|g\|^2_{L^\infty([0, T];\,H^N_\ell(\RR^6))}+ ||g||^2_{L^2([0,
T];\,\cB^N_\ell(\RR^6))}\leq Ce^{C\,
T}\Big(\|g_0\|^2_{H^N_\ell(\RR^6)}+\epsilon_0^2 T\Big),
\end{equation*}
for a constant $C>0$ depending only on $N, \ell$ .
\end{theo}

And this
yields the local existence of solution by the contraction mapping theorem
through the
standard argument. Therefore, we omit the details for the brevity of the
paper.

\section{Global Existence}

In this section, we derive a global energy estimate for  the solution
in the weighted function spaces. In the soft potential case considered in Part I,
we could
obtain two types of global energy estimates, one for only
 $x$ derivatives without requiring any weight in the variable
$v$ and one for both $x$ and $v$ derivatives with weight in $v$ whose order
varies with the order of $v$ derivative. On the other hand,
in the hard potential case,  the energy estimate can be closed only when
 both $x$ and $v$ derivatives
are taken into account together with weight in $v$. This is due to the
upper bound estimate on the nonlinear collision operator given in
Section 2 where some weighted norms are used.  However, the order
of  weight can be fixed in contrast to the case of soft potential.


Set

\begin{align*}
\cE_{N,\ell}&= \|g\|_{H^N_\ell(\RR^6)}^2\sim
\|g_1\|^2_{{ H}^N_\ell(\RR^6)}+\|g_2\|^2_{{ H}^N_\ell(\RR^6)}\sim \|\cA\|_{H^{N}(\RR^3)}^2
+\|g_2\|^2_{{ H}^N_\ell(\RR^6)},
\\ \cD_{N,\ell}&=\|\nabla_x g_1\|^2_{{ H}^{N-1}(\RR^6)} +\|g_2\|_{{\cB}^N_\ell(\RR^6)}^2  \sim\|\nabla_x\cA\|^2_{H^{N-1}(\RR^3)}+\|g_2\|_{{\cB}^N_\ell(\RR^6)}^2,
\end{align*}
where $\cA=(a,b,c)$ with $\pP g=g_1=(a+v\cdot b+|v|^2 c)\mu^{\frac 12}$, and $g_2=g-\pP g$.

Let  $g=g(t,x,v)$ be a smooth solution to
the Cauchy problem \eqref{4-3-1.3}.
The main goal of this section is to establish
\begin{prop} \label{energy}{\rm \bf (Energy Estimate) \ }
Assume
$0<s<1$ and $2s+\gamma> 0$, and
\label{energyineq}
let $N\ge 6$, $\ell>3/2+2s+\gamma$. Then,
\[
\frac{d}{dt}\cE_{N,\ell}+\cD_{N,\ell}\lesssim \cE_{N,\ell}^{1/2}\cD_{N,\ell},
\]
holds as long as the solution $g$ exists.
\end{prop}

With this proposition, the standard continuity argument and
the local existence assure the
global existence of solution when the intial data
$g_0$ satisfies that $\cE_{N,\ell}(0)$ is sufficiently small.
And the above energy estimate will be obtained by using the
coercivity, upper bound and commutator estimates through the macro-micro
decomposition introduced in \cite{guo-1} as follows.

\subsection{ Macroscopic energy estimate.}
As in  \cite{guo-1},  the macroscopic component $\cA=(a,b,c)
$ satisfies
\begin{equation}\label{macroeq}
\left\{
\begin{array}{lrlrl}
& v_i | v|^2 \mu^{1/2} :&\quad&\nabla_x c &= -\partial_t r_c+l_c + h_c,
\\
 &v^2_i \mu^{1/2}:&\quad&\partial_t c +\partial_ib_i &= -\partial_t r_i+l_i + h_i ,
\\
& v_iv_j \mu^{1/2}:& &\partial_ib_j + \partial_j b_i &
= -\partial_t r_{ij}+l_{ij} + h_{ij} , \quad i\neq j,
\\
& v_i \mu^{1/2} :&&\partial_t b_i + \partial_i a &
= -\partial_t r_{bi}+l_{bi} + h_{bi},
\\
& \mu^{1/2} :&& \partial_t a &= -\partial_t r_a+l_a + h_a,
\end{array}
\right.
\end{equation}
where
\begin{align*}
&r=(g_2,e)_{L^2_v},\qquad l=-(v\cdot\nabla_x g_2+\mathcal{L} g_2,e)_{L^2_v},
\qquad h=(\Gamma(g,g),e)_{L^2(\RR^3_v)},
\end{align*}
stand for $r_c, \cdots, h_a$, while
\begin{align*}
 e\in \text{span}
\lbrace v_i | v|^2 \mu^{1/2} , v^2_i \mu^{1/2} ,
v_iv_j \mu^{1/2}, v_i \mu^{1/2} , \mu^{1/2} \rbrace.
\end{align*}

Same as Lemma 7.2 in \cite{amuxy3}, we have the following property on $\cA$.
\begin{lemm}\label{abc2}
Let $\p^\alpha=\p^\alpha_x$, $\alpha=\alpha_1+\alpha_2\in\NN^3, |\alpha|\le N, N\ge 3$.  Then,
\begin{equation*}
\|(\p^{\alpha_1} \cA)(\p^{\alpha_2} \cA)\|_{L^2_x}\le \|\nabla_x \cA
\|_{H^{N-1}_{x}}\|\cA\|_{H^{N-1}_{x}}.
\end{equation*}
\end{lemm}

The following lemma is similar to Lemma 7.3 in \cite{amuxy3} with
some modification regarding to  the hard potential assumption. Here, we include its
proof for the completeness.

\begin{lemm} \label{rlh}
Let $\p^\alpha=\p^\alpha_x  , \p_i=\p_{x_i}$, $|\alpha|\le N-1, N\ge 3$. Then, one has
\begin{align}\label{Dr}
&\|\p_i\p^\alpha r\,\|_{L^2_x}+\|\p^\alpha   l\,\|_{L^2_{x}}
\lesssim
  \|g_2\|_{H^N(\RR^3_x; L^2(\RR^3_v))}\equiv A_1,
  \\&\label{Dh}
   \|\p^\alpha   h\|_{L^2_x}
 \lesssim \|\nabla_x \cA\|_{H^{N-2}_{x}}\|\cA\|_{H^{N-1}_{x}}
+\|\cA \|_{H^N_x}\|g_2\|_{\cB^N_0(\RR^6)}
\\&\hspace{1cm}+\notag
\|g_2\|_{H^N_{s+\gamma/2}(\RR^6)}
\| g_2\|_{\cB^N_0(\RR^6)}
\equiv A_2.
\end{align}
\end{lemm}
\noindent
\begin{proof}
By the Cauchy-Schwarz inequality,
\begin{align*}
\|\p_i\p^\alpha r\,\|_{L^2(\RR^3_x)}
&=
\| (\p^\alpha \nabla_x g_2, e)_{L^2(\RR^3_v)}\|_{L^2(\RR^3_x)}
\le
  \|\p^\alpha  \nabla_x g_2\|_{L^2(\RR^6)}
\le
\|g_2\|_{H^N(\RR^3_x,\, L^2(\RR^3_v))} ,
\end{align*}
and
 \begin{align*}
\|\p^\alpha   l\,\|_{L^2_{x}}&\le
\| (\nabla_x\p^\alpha g_2, ve)_{L^2(\RR^3_v)}\|_{L^2(\RR^3_x)}+
\|(\p^\alpha g_2, \mathcal{L}^*e)_{L^2(\RR^3_v)}\|_{L^2(\RR^3_x)}
\\&\le \|\p^\alpha  \nabla_x g_2\|_{L^2(\RR^6)}+\|\p^\alpha   g_2\|_{L^2(\RR^6)}
\le \|g_2\|_{H^N(\RR^3_x,\, L^2(\RR^3_v))}.
 \end{align*}
Then \eqref{Dr} follows because $H^N_{\gamma/2+s}(\RR^6)\subset \cB^N_0(\RR^6)$
holds by virtue of Proposition  \ref{partI-prop2.2}.
We shall prove \eqref{Dh} as follows. By  Proposition \ref{prop5.1},
\begin{align*}
 |\p_x^\alpha h|&\le\sum_{\alpha_1+\alpha_2=\alpha}
|(\Gamma(\p_x^{\alpha_1}g,\p_x^{\alpha_2}g),e)_{L^2_v} |
\lesssim
 \sum_{\alpha_1+\alpha_2=\alpha}
 ||\p_x^{\alpha_1} g||_{L^2_{s+\gamma/2}}||| \ \p_x^{\alpha_2}g|||\ |||e|||
\\&\lesssim
 \sum_{\alpha_1+\alpha_2=\alpha}(|\p_x^{\alpha_1}\cA|
+||\p_x^{\alpha_1} g_2||_{L^2_{s+\gamma/2}})
 (|\p_x^{\alpha_2}\cA|
+|||\p_x^{\alpha_2} g_2|||).
\end{align*}
Hence
\begin{align*}
\|\p_x^\alpha h\|_{L^2_x}&\lesssim
 \|\ |\p_x^{\alpha_1}\cA|^2\ |\p_x^{\alpha_1}\cA|\  \|_{L^2_x}+H,
\end{align*}
where
\begin{align*}
H&=\|\ |\p_x^{\alpha_1}\cA |\ |||\p_x^{\alpha_2} g_2|||\  \|_{L^2_x}
+\|\
 ||\p_x^{\alpha_1} g_2||_{L^2_{s+\gamma/2}}|\p^{\alpha_2}\cA|\  \|_{L^2_x}
\\&+\|\
 ||\p_x^{\alpha_1} g_2||_{L^2_{s+\gamma/2}}
|||\p_x^{\alpha_2} g_2|||\  \|_{L^2_x}.
\end{align*}
The first term on the right hand of the above inequality
can be  evaluated by using Lemma \ref{abc2}.
As for $H$, when $|\alpha_1|=0,1$, by using Proposition \ref{partI-prop2.2}, we have
\begin{align*}
H&\lesssim \|\p_x^{\alpha_1}\cA \|_{H^2}\|g_2\|_{\cB^N_0(\RR^6)}
+\|\p_x^{\alpha_1}g_2\|_{H^2_{s+\gamma/2}(\RR^6)}
\|\cA \|_{H^N_x}+\|\p_x^{\alpha_1}g_2\|_{H^2_{s+\gamma/2}(\RR^6)}
\| g_2\|_{\cB^N_0(\RR^6)}
\\&
\lesssim
\|\cA \|_{H^N_x}\|g_2\|_{\cB^N_0(\RR^6)}
+
\|g_2\|_{H^N_{s+\gamma/2}(\RR^6)}
\| g_2\|_{\cB^N_0(\RR^6)}
.
\end{align*}
Similarly, when $2\le |\alpha_1|\le N$, we have
\begin{align*}
H&\lesssim \|\p_x^{\alpha_1}\cA \|_{L^2}\|\p_x^{\alpha_2}g_2\|_{H^2(\RR^3;\cB^0_0(\RR^3)}
+\|\p_x^{\alpha_1}g_2\|_{H^0_{s+\gamma/2}(\RR^6)}
(\|\cA \|_{H^N_x}+
\| g_2\|_{\cB^N_0(\RR^6)})
\\&
\lesssim
\|\cA \|_{H^N_x}\|g_2\|_{\cB^N_0(\RR^6)}
+
\|g_2\|_{H^N_{s+\gamma/2}(\RR^6)}
\| g_2\|_{\cB^N_0(\RR^6)}.
\end{align*}
Thus, the proof of the lemma is completed.
\end{proof}

The following lemma about the energy estimate on the macroscopic component
which is also similar to the corresponding one in \cite{amuxy3} for
the Maxwellian molecule.

\begin{lemm}\label{pcA}
For $|\alpha|\le N-1$, we have
\begin{align}\label{pabc}
\|\nabla_x \p^\alpha \cA\|_{L^2(\RR^3_x)}^2
&\lesssim
-\frac{d}{dt}\Big\{(\p^\alpha r,\nabla_x \p^\alpha (a, - b, c))_{L^2(\RR^3_x)}
+(\p^\alpha b, \nabla_x \p^\alpha a)_{L^2(\RR^3_x)}\Big\}
\\& +\notag
\|g_2\|_{H^N(\RR^3_x;L^2(\RR^3_v))}^2+E_{N,1}D_{N,0},
\end{align}
where
\begin{align*}
&E_{N,\ell}=\|\cA\|_{H^{N}(\RR^3)}^2+\|g_2\|_{H^N(\RR^3_x;L^2_l(\RR^3_v))}^2,
\quad
D_{N,\ell}=\|\nabla_x\cA\|_{H^{N-1}(\RR^3)}^2+\sum_{|\alpha|\le N}\|\partial^\alpha_x g_2\|_{\cB^0_\ell(\RR^6)}.
\end{align*}
\end{lemm}
\begin{proof}
(a) Estimate on $\nabla_x\p^\alpha  a$. Let $A_1,A_2$ be those defined
 in Lemma \ref{rlh}.
{}From \eqref{macroeq} (iv),
\begin{align*}
 \|\nabla_x \p^\alpha a\|_{L^2(\RR^3_x)}^2
&=(\nabla_x \p^\alpha a,\nabla_x \p^\alpha a)_{L^2(\RR^3_x)}
\\&= (\p^\alpha (-\p_t b-\p_tr+l+h),\nabla_x \p^\alpha a)_{L^2(\RR^3_x)}
\\&\le R_1+C_\eta (A_1^2+A_2^2)+\eta \|\nabla_x \p^\alpha a\|_{L^2(\RR^3_x)}^2.
\end{align*}
Here,
\begin{align*}
R_1&=-(\p^\alpha \p_tb+\p^\alpha \p_tr,\nabla_x \p^\alpha a)_{L^2(\RR^3_x)}
\\&=-\frac{d}{dt}(\p^\alpha (b+r),\nabla_x \p^\alpha a)_{L^2(\RR^3_x)}+
(\nabla_x\p^\alpha (b+r), \p_t \p^\alpha a)_{L^2(\RR^3_x)}
\\&
\le
-\frac{d}{dt}(\p^\alpha (b+r),\nabla_x \p^\alpha a)_{L^2(\RR^3_x)}+
C_\eta ( \|\nabla_x\p^\alpha  b\|_{L^2(\RR^3_x)}^2+A_1^2)
+\eta
\|\p_t\p^\alpha  a\|^2_{L^2(\RR^3_x)}.
\end{align*}
(b) Estimate on $\nabla_x\p^\alpha  b$. {}From \eqref{macroeq} (iii) and (ii),
\begin{align*}
\Delta_x&\p^\alpha b_i+\p^2_{i}\p^\alpha   b_i=
{\sum_{j \ne i} \p_j \p^{\alpha}( \p_j b_i + \p_i b_j)
+ \p_i \p^\alpha( 2 \p_i b_i - \sum_{j\ne i} \p_i b_j)
}
\\
&\hskip2cm =
\p_i\p^\alpha  (-\p_t r +l+h),
\\
&\|\nabla_x\p^\alpha  b\|_{L^2(\RR^3_x)}^2+\|\p_i\p^\alpha  b\|_{L^2(\RR^3_x)}^2=
-(\Delta_x\p^\alpha b_i+\p^2_{i}\p^\alpha   b_i, \p^\alpha  b)_{L^2(\RR^3_x)}
= R_2+R_3+R_4,
\end{align*}
where
\begin{align*}
R_2&=(\p_t\p^\alpha   r,\p_i\p^\alpha  b)_{L^2(\RR^3_x)}
=\frac{d}{dt}(\p^\alpha   r,\p_i\p^\alpha  b)_{L^2(\RR^3_x)}+
(\p_i\p^\alpha   r,\p_t\p^\alpha  b)_{L^2(\RR^3_x)}
\\&\hspace*{1cm}\le
\frac{d}{dt}(\p^\alpha   r,\p_i\p^\alpha  b)_{L^2(\RR^3_x)}+
C_\eta A_1^2
+\eta
\|\p_t\p^\alpha  b\|^2_{L^2(\RR^3_x)},
\\
R_3&=-(\p^\alpha   l,\p_i\p^\alpha  b)_{L^2(\RR^3_x)}\le
C_\eta A_1^2
+
\eta
\|\p_i\p^\alpha  b\|^2_{L^2(\RR^3_x)},
\\R_4&=-(\p^\alpha   h,\p_i\p^\alpha  b)_{L^2(\RR^3_x)}\le
C_\eta A_2^{{2}}
+
\eta
\|\p_i\p^\alpha  b\|^2_{L^2(\RR^3_x)}.
\end{align*}
(c) Estimate on $\nabla_x\p^\alpha  c$. {}From \eqref{macroeq} (i),
 \begin{align*}
 \|\nabla_x \p^\alpha &c\|_{L^2(\RR^3_x)}^2
=(\nabla_x \p^\alpha c,\nabla_x \p^\alpha c)_{L^2(\RR^3_x)}
= (\p^\alpha (-\p_tr+l+h),\nabla_x \p^\alpha c)_{L^2(\RR^3_x)}
\\&
\le R_5 
+C_\eta(A_1^2+A_2^2)+\eta \|\nabla_x \p^\alpha c\|_{L^2(\RR^3_x)}^2,
\end{align*}
where
\begin{align*}
R_5&=-(\p^\alpha \p_tr,\nabla_x \p^\alpha c)_{L^2(\RR^3_x)}=
-\frac{d}{dt}(\p^\alpha r,\nabla_x \p^\alpha c)_{L^2(\RR^3_x)}+
(\nabla_x\p^\alpha r, \p_t \p^\alpha c)_{L^2(\RR^3_x)}
\\&\quad
\le
-\frac{d}{dt}(\p^\alpha r,\nabla_x \p^\alpha c)_{L^2(\RR^3_x)}+
C_\eta A_1^2 
+\eta
\|\p_t\p^\alpha  c\|^2_{L^2(\RR^3_x)}.
\end{align*}
(d) Estimate on $\p_t\p^\alpha  \cA$. We directly have
\begin{align*}
\|\p_t\p^\alpha & \cA\|_{L^2(\RR^3_x)}=\|\p^\alpha \p_t \pP g\|_{L^2(\RR^6_{x, v})}
\\& \notag
=\|\p^\alpha  \pP\Big(-v\cdot\nabla_xg-\mathcal{L}g+\Gamma(g,g)\Big)\|_{L^2(\RR^6_{x, v})}
\\&=\|\p^\alpha  \pP(v\cdot\nabla_xg)\|_{L^2(\RR^6_{x, v})}
\le \|\nabla_x\p^\alpha  \cA\|_{L^2(\RR^3_x)}+
\|\nabla_x\p^\alpha g_2\|_{L^2(\RR^6_{x, v})}.\notag
\end{align*}
Combining all the above estimates and
taking $\eta>0$ sufficiently small, we deduce
\begin{align*}
\|\nabla_x \p^\alpha \cA\|_{L^2(\RR^3_x)}^2
&\lesssim
-\frac{d}{dt}\Big\{(\p^\alpha r,\nabla_x \p^\alpha (a, -b,c))_{L^2(\RR^3_x)}
+(\p^\alpha b, \nabla_x \p^\alpha a)_{L^2(\RR^3_x)}\Big\}
\\&\quad +\notag
A_1^2+A_2^2+
 \|\nabla_x\p^\alpha g_2\|_{L^2(\RR^6_{x, v})}^2.
\end{align*}
Finally, by choosing $|\alpha|\le N-1$ and using Lemma \ref{rlh}, we obtain
\begin{align*}
A_1^2&+A_2^2+
 \|\nabla_x\p^\alpha g_2\|_{L^2(\RR^6_{x, v})}^2
\lesssim \|g_2\|_{H^N(\RR^3_x;L^2(\RR^3_v))}^2+E_{N,1}D_{N,0},
\end{align*}
which  completes the proof of the lemma.
\end{proof}


\subsection{ Microscopic energy. }

The energy estimate on the microscopic component will be given in
two parts, that is, one without weight and another one with weight
as follows.

\subsection*{Microscopic energy  estimate without weight}\label{section7.2}
In this subsection,
we shall prove the  following estimate with only
 $x$-derivatives of the solution.
\begin{lemm} \label{micro1}
Let $ N\ge 3$. Then,
\begin{equation}\label{microenergy11}
\frac{d}{dt}E_{N,0}+D_{N,0}
\lesssim
E_{N,{s+\gamma/2}}^{1/2}D_{N,0}.
\end{equation}
\end{lemm}
Notice that this is not a closed estimate because of the presence of  $E_{N,s+\gamma/2}$
on the right hand side. Since $s+\gamma/2>0$, this is exactly why we can not prove
global existence with only differentiation in $x$ variable.
\vspace{0.5cm}

For the proof,
let $\alpha\in\NN^3$, $|\alpha|\le N$, and apply $\p^\alpha=\p^\alpha_x$ to \eqref{4-3-1.3} to have,
\begin{equation*}
\p_t(\p^\alpha g)+ v\cdot\nabla_x (\p^\alpha g)
+ \cL(\p^\alpha g)=\p^\alpha\Gamma(g,\,
g).
\end{equation*}
Then take the $L^2(\RR^6_{x, v})$ inner product of the above
equation with $\p^\alpha g$. By Proposition \ref{prop2.1}, we have
\begin{equation}\label{microenergy1}
\frac{d}{dt}\|\p^\alpha g\|^2_{L^2(\RR^6_{x, v})}
+ ||| \p^\alpha g_2 |||^2_{\cB^0_0(\RR^6)}\lesssim J_\alpha,
\end{equation}
where
\begin{align*}
J_\alpha&=(\p^\alpha\Gamma(g,\,g), \p^\alpha g)_{L^2(\RR^6)}
=\sum_{i,j=1}^2(\p^\alpha\Gamma(g_i,\,g_j), \p^\alpha g_2)_{L^2(\RR^6)}
\\&=\sum_{i,j=1}^2 J^{(ij)}.
\end{align*}

Remember that we are dealing with the case when $0<s<1, \gamma +2s> 0$.
Firstly, consider $J^{(11)}$.  For $\psi_j\in\mathcal{N}$,
\begin{align*}
|J^{(11)}|&
\lesssim \int_{\RR^3}|\p^\alpha \cA^2|\ |(\Gamma(\psi_j,\psi_k), \p^\alpha g_2)_{L^2(\RR^3)}|dx.
\end{align*}
According to Theorem 1.2  of \cite{amuxy4-1}, we have,
\begin{align*}
|(\Gamma(\psi_j,\psi_k),
  \p^\alpha g_2)_{L^2(\RR^3_v)}|
  &\lesssim \Big(\|\psi_j\|_{L^2_{s+\gamma/2}}
|||\psi_k|||+\|\psi_k\|_{L^2_{s+\gamma/2}}|||\psi_j|||\Big)|||\p^\alpha g_2|||
\\&
\lesssim
|||\p^\alpha g_2|||,
\end{align*}
and hence
\begin{align*}
|J^{(11)}|&
\lesssim \|\p^\alpha \cA^2\|_{L^2(\RR^3)}
||\p^\alpha g_2||_{\cB^{0}_0(\RR^6)}
\lesssim
\|\nabla \cA\|_{H^{N-1}_x(\RR^3)}\|\cA\|_{H^{N-1}_x(\RR^3)}
||\p^\alpha g_2||_{\cB^{0}_0(\RR^6)}
\\&\lesssim
E_{N,0}^{1/2}D_{N,0}.
\end{align*}

On the other hand,
\begin{align*}
 J^{(12)}&\sim\int_{\RR^3_x}(\p_x^{\alpha_1}\cA)
(\Gamma(\psi_k, \p_x^{\alpha_2}g_2), \p_x^{\alpha} g_2)_{L^2(\RR^3_v)}dx.
\end{align*}
In view of  Theorem 1.2  of \cite{amuxy4-1}, we obtain
\begin{align*}
|(\Gamma(\psi_k,&\p_x^{\alpha_2}g_2), \p_x^{\alpha} g_2)_{L^2(\RR^3_v)}|
\lesssim\Big\{ \|\psi_k\|_{L^2_{s+\gamma/2}}|||\p_x^{\alpha_2}g_2|||
+ \|\p_x^{\alpha_2}g_2\|_{L^2_{s+\gamma/2}}|||\psi_k|||\\
&\quad + \min \big ( \|\psi_k\|_{L^2}\|\p_x^{\alpha_2}g_2\|_{L^2_{s+\gamma/2}}
\,,\,\|\psi_k\|_{L^2_{s+\gamma/2}}\|\p_x^{\alpha_2}g_2\|_{L^2}
\big)\Big\}|||\p_x^{\alpha}g_2|||\notag
\\&
\lesssim\Big(|||\p_x^{\alpha_2}g_2|||
+  \|\p_x^{\alpha_2}g_2\|_{L^2_{s+\gamma/2}}\Big)|||\p_x^{\alpha}g_2|||\notag
\\&
\lesssim |||\p_x^{\alpha_2} g_2||| \ |||\p_x^{\alpha} g_2|||.
  \end{align*}
 Here and hereafter,
 we will use freely that
 \begin{align*}
 \|g\|_{L^2}\lesssim \|g\|_{L^2_{s+\gamma/2}}\lesssim |||g|||.
\end{align*}
The first inequality above holds since we assume $s+\gamma/2>0$ and the second one
follows
from Theorem 1.2  of \cite{amuxy4-1}.
By using Proposition \ref{prop5.1},
  \begin{align}\notag
| J^{(12)}|&
 \lesssim\int_{\RR^3_x}\|\p_x^{\alpha_1}\cA|\ |||\p_x^{\alpha_2}g_2||| \ |||\p_x^{\alpha} g_2|||dx
 \lesssim
E_{N,0}^{1/2}D_{N,0}.
\end{align}

A similar argument  applies to
\begin{align*}
 J^{(21)}&\sim\int_{\RR^3_x}(\p_x^{\alpha_2}\cA)
(\Gamma(\p_x^{\alpha_1}g_2, \varphi_k), \p_x^{\alpha} g_2)_{L^2(\RR^3_v)}dx.
\end{align*}
 In fact, Theorem 1.2  of \cite{amuxy4-1} gives,
\begin{align*}
\Big|\Big(\Gamma(\p_x^{\alpha_1}g_2, \varphi_k), \p_x^{\alpha} &g_2)_{L^2(\RR^3_v)}\Big|
\lesssim\Big( \|\p_x^{\alpha_1}g_2\|_{L^2_{s+\gamma/2}}|||\psi_k|||
+\|\psi_k\|_{L^2_{s+\gamma/2}}|||\p_x^{\alpha_1}g_2|||
\\
&\quad { + \min \big ( \|\p_x^{\alpha_1}g_2\|_{L^2} \ \|\psi_k\|_{L^2_{s+\gamma/2}}
\,,\,\|\p_x^{\alpha_1}g_2\|_{L^2_{s+\gamma/2}} \|\psi_k\|_{L^2
 }
\big)}\Big)|||\p_x^{\alpha}g_2|||
\\
&\lesssim( \|\p_x^{\alpha_1}g_2\|_{L^2_{s+\gamma/2}}
+|||\p_x^{\alpha_1}g_2|||
)|||\p_x^{\alpha}g_2|||
\\&\lesssim
 \ |||\p_x^{\alpha_1}g_2||| \ |||\p_x^{\alpha}g_2|||.
\end{align*}
Therefore,  similar to $J^{(12)}$, we have
\[
|J^{(21)}| \lesssim
E_{N,0}^{1/2}D_{N,0}.
\]

For the estimation on $J^{(22)}$,
 from  Theorem 1.2 of \cite{amuxy4-1} again, we have
\begin{align*}
|(\Gamma(\p_x^{\alpha_1}g_2, &\p_x^{\alpha_2}g_2), \p_x^{\alpha} g_2)_{L^2(\RR^3)}|
\lesssim
\Big\{\|\p_x^{\alpha_1}g_2\|_{L^2_{s+\gamma/2}}|||\p_x^{\alpha_2}g_2|||
+|||\p_x^{\alpha_1}g_2||| \ \|\p_x^{\alpha_2}g_2\|_{L^2_{s+\gamma/2}}
\\&
{ + \min \big ( \|\p_x^{\alpha_1}g_2\|_{L^2} \ \|\p_x^{\alpha_2}g_2\|_{L^2_{s+\gamma/2}}
\,,\,\|\p_x^{\alpha_1}g_2\|_{L^2_{s+\gamma/2}} \|\p_x^{\alpha_2}g_2\|_{L^2}
\big)} \Big\}
 \
||| \p_x^{\alpha} g_2|||
\\&
\lesssim
\Big(\|\p_x^{\alpha_1}g_2\|_{L^2_{s+\gamma/2}}|||\p_x^{\alpha_2}g_2|||
+|||\p_x^{\alpha_1}g_2||| \ \|\p_x^{\alpha_2}g_2\|_{L^2_{s+\gamma/2}}
 \Big)
 \
||| \p_x^{\alpha} g_2|||.
\end{align*}
Firstly, suppose that $|\alpha_1|\le N-2$. Then
\begin{align*}
|J^{(22)}|&=|(\Gamma(\p_x^{\alpha_1}g_2, \p_x^{\alpha_2}g_2), \p_x^{\alpha} g_2)_{L^2(\RR^6)}|
\\&
\lesssim
\|\p_x^{\alpha_1}g_2\|_{L^\infty_x(L^2_{s+\gamma/2})}\int_{\RR^3}|||\p_x^{\alpha_2}g_2|||\
||| \p_x^{\alpha} g_2|||dx
+\|\ |||\p_x^{\alpha_1}g_2||| \ \|_{L^\infty_x}
\int_{\RR^3}||\p_x^{\alpha_2}g_2||_{L^2_{s+\gamma/2}}\
||| \p_x^{\alpha} g_2|||dx
\\&
\lesssim
E_{N,s+\gamma/2}^{1/2}D_{N,0}.
\end{align*}
Similary, when $|\alpha_1|> N-2$, we have  $|\alpha_2|\le 1$ so that
\begin{align*}
|J^{(22)}|&
\lesssim
\|\p_x^{\alpha_2}g_2\|_{L^\infty_x(L^2_{s+\gamma/2})}\int_{\RR^3}|||\p_x^{\alpha_1}g_2|||\
||| \p_x^{\alpha} g_2|||dx
+\|\ |||\p_x^{\alpha_2}g_2||| \ \|_{L^\infty_x}
\int_{\RR^3}||\p_x^{\alpha_1}g_2||_{L^2_{s+\gamma/2}}\
||| \p_x^{\alpha} g_2|||dx
\\&
\lesssim
E_{N,s+\gamma/2}^{1/2}D_{N,0}.
\end{align*}

Taking the summation of \eqref{microenergy1} over $|\alpha|\le N, N\ge 3$
gives \eqref{microenergy11}.
\subsection*{Microscopic energy estimate with weight.}
In order to close the estimate \eqref{microenergy11}, we need the estimates of
$x$-$v$ derivatives of the solution with weight.
Firstly, let
$
\p^\alpha_\beta=\p^\alpha_x\p^\beta_v,  |\alpha+\beta |\le N, N\ge 6,
$
and apply $W_\ell\p^\alpha_\beta(\iI-\pP)$ to \eqref{4-3-1.3}.
We have
\begin{align*}
 \p_t (W_\ell\p^\alpha_\beta g_2)+ v\cdot\nabla_x
(W_\ell\p^\alpha_\beta g_2) + \cL _1(W_\ell\p^\alpha_\beta g_2)=\cM^{\alpha,\beta},
\end{align*}
where, with $e_i\in \NN^n, |e_i|=1$,
\begin{align*}
\cM^{\alpha,\beta}&=W_\ell\p^\alpha_\beta \Gamma(g,\, g)
+W_\ell\p^\alpha_\beta[\pP, v \cdot\nabla_x]g
-\sum_{i=1}^N\p_{x_i}(W_\ell\p^\alpha_{\beta-e_i} g_2)
-[W_\ell\p^\alpha_\beta, \cL_1]g_2-W_\ell\p^\alpha_\beta\cL_2g_2
\\&=\cM_1+\cM_2+\cM_3+\cM_4+\cM_5.
\end{align*}
Take the $L^2(\RR^6_{x, v})$ inner product of this equation
with $W_\ell\p^\alpha_\beta g_2$ to deduce
\begin{equation}\label{energyweight}
 \frac{d}{dt}\| W_\ell\p^\alpha_\beta g_2\|^2_{L^2(\RR^6)}+D\lesssim M,
\end{equation}
where $D$ is the dissipation rate given by
\begin{align*}
 D&=\int_{\RR^3}
||| (\iI-\pP)W_\ell\p^\alpha_\beta g_2 |||^2dx
\ge|| W_\ell\p^\alpha_\beta g_2 ||_{\cB^0_0(\RR^3)}
-C||  g_2 ||_{H^N_0(\RR^6)}^2.
\end{align*}
Here, we have used
\begin{align*}
|||\pP W_\ell \p^{\alpha}_{\beta}g_2|||
\lesssim||\p^\alpha g_2||_{L^2_v}.
\end{align*}
And $M$ is defined by
\[
M=
\sum_{j=1}^5(\cM_j, W_\ell\p^\alpha_\beta g_2)_{L^2(\RR^6)}
=\sum_{j=1}^5M_j.
\]

Firstly,
note that from Propositions 2.3 and 2.5, we have
\begin{align}\label{a11}
|(W_\ell& \Gamma(\p_{\beta_1}f,\p_{\beta_2}g), h)_{L^2(\RR^3)}|
\lesssim
\Big(\|W_{s+\gamma/2}\p_{\beta_1}f\|_{L^2(\RR^3)}|||W_{\ell} \p_{\beta_2}g|||
\\&\hspace{1cm}
+||W_{s+\gamma/2} \p_{\beta_2}g||_{L^2(\RR^3)} \ |||W_{\ell}\p_{\beta_1}f|||\notag
 \Big)|||h|||.
 \end{align}

Write
\begin{align*}
M_1&=\sum_{i,j=1,2}(W_\ell\p^\alpha_\beta \Gamma(g_i,\,g_j), W_\ell\p^\alpha_\beta g_2)_{L^2(\RR^6)}
=\sum_{i,j=1,2}M_{1ij}.
\end{align*}
We have
\begin{align*}
M_{111}&=(W_\ell\p^\alpha_\beta \Gamma(g_1,\,g_1), W_\ell\p^\alpha_\beta g_2)_{L^2(\RR^6)}
\\&\sim \int_{\RR^3}
(\p^{\alpha_1}\cA)(\p^{\alpha_2}\cA)
(W_\ell\p_\beta\Gamma(\psi_j,\psi_k), W_\ell\p^\alpha_\beta g_2)_{L^2(\RR^3)}dx.
\end{align*}
Recall that by Leibnitz formula, the differentiation on $\Gamma$ involves
the nonlinear operators $\Gamma$ and $\cT$. Since these two operators
 share the same upper bound and commutator properties, for brevity, we only consider
 the nonlinear operator $\Gamma$.
 By using  \eqref{a11},  since $\psi$ is a  function with an
exponential decay factor, we obtain for $|\beta_1|+|\beta_2|\le |\beta|$,
\begin{align*}
&||(W_\ell \Gamma(\p_{\beta_1}\psi_j,\p_{\beta_2}\psi_k), W_\ell\p^\alpha_\beta g_2)_{L^2(\RR^3)}|
\lesssim |||W_\ell\p^\alpha_\beta g_2|||.
 \end{align*}
 Therefore,
 \begin{align*}
 M_{111}&\sim \int_{\RR^3}
(\p^{\alpha_1}\cA)(\p^{\alpha_2}\cA)|||W_\ell\p^\alpha_\beta g_2|||
dx
\\&\lesssim \|\nabla_x\cA\|_{H^{N-1}(\RR^3)}\|\cA\|_{H^{N}(\RR^3)}\|g_2\|_{\cB^N_\ell(\RR^6)}.
 \end{align*}

Now consider
\begin{align*}
M_{112}&=(W_\ell\p^\alpha_\beta \Gamma(g_1,g_2), W_\ell\p^\alpha_\beta g_2)_{L^2(\RR^6)}
\\&\sim \int_{\RR^3}
(\p^{\alpha_1}\cA)
(W_\ell\p_\beta\Gamma(\psi_j,\p^{\alpha_2}g_2), W_\ell\p^\alpha_\beta g_2)_{L^2(\RR^3)}dx
.
\end{align*}
By  \eqref{a11},
\begin{align*}
&|(W_\ell \Gamma(\p_{\beta_1}\psi_j,\p^{\alpha_2}_{\beta_2}g_2), W_\ell\p^\alpha_\beta g_2)_{L^2(\RR^3)}|
\\&\lesssim
\Big(\|W_{s+\gamma/2}\p_{\beta_1}\psi_j\|_{L^2(\RR^3)}|||W_{\ell} \p^{\alpha_2}_{\beta_2}g_2|||
+
||W_{s+\gamma/2} \p^{\alpha_2}_{\beta_2}g_2||_{L^2(\RR^3)} \ |||W_{\ell}\p_{\beta_1}\psi_j|||
 \Big)|||W_\ell \p^{\alpha}_{\beta}g_2|||
 \\&\lesssim
 |||W_{\ell} \p^{\alpha_2}_{\beta_2}g_2|||\ |||W_\ell \p^{\alpha}_{\beta}g_2|||,
\end{align*}
where we have used
\[
||W_{s+\gamma/2} \p^{\alpha_2}_{\beta_2}g_2||_{L^2(\RR^3)} \lesssim
|||W_{\ell} \p^{\alpha_2}_{\beta_2}g_2|||.
\]
Thus,
\begin{align*}
M_{112}\lesssim
\|\cA\|_{H^N(\RR^3)}\|g_2\|_{\cB^N_\ell}^2.
\end{align*}

Next, notice that
\begin{align*}
M_{121}&=(W_\ell\p^\alpha_\beta \Gamma(g_2,g_1), W_\ell\p^\alpha_\beta g_2)_{L^2(\RR^6)}
\\&\sim \int_{\RR^3}
(\p^{\alpha_2}\cA)
(W_\ell\p_\beta\Gamma(\p^{\alpha_1} g_2,\psi_j), W_\ell\p^\alpha_\beta g_2)_{L^2(\RR^3)}dx
.
\end{align*}
As above,   \eqref{a11}
yields
\begin{align*}
&|(W_\ell \Gamma(\p^{\alpha_1}_{\beta_1} g_2,\p_{\beta_2}\psi_j), W_\ell\p^\alpha_\beta g_2)_{L^2(\RR^3)}|
\\&\lesssim
\Big(\|W_{s+\gamma/2}\p^{\alpha_1}_{\beta_1} g_2\|_{L^2(\RR^3)}
|||W_{\ell} \p_{\beta_2}\psi_j|||
+||W_{s+\gamma/2} \p_{\beta_2}\psi_j||_{L^2(\RR^3)} \ |||W_{\ell}\p^{\alpha_1}_{\beta_1}g_2|||
 \Big)|||W_\ell \p^{\alpha}_{\beta}g_2|||
\\&\lesssim
\Big(\|W_{s+\gamma/2}\p^{\alpha_1}_{\beta_1} g_2\|_{L^2(\RR^3)}
+ |||W_{\ell}\p^{\alpha_1}_{\beta_1}g_2|||
 \Big)
|||W_\ell \p^{\alpha}_{\beta}g_2|||
\\&\lesssim
|||W_{\ell} \p^{\alpha_1}_{\beta_1}g_2|||\ |||W_\ell \p^{\alpha}_{\beta}g_2|||,
 \end{align*}
where we have used $\|W_{s+\gamma/2} g\|_{L^2(\RR^3)}\lesssim |||W_\ell g|||$.
 Consequently,
  \begin{align*}
 M_{121}
\lesssim
\|\cA\|_{H^N(\RR^3)}\|g_2\|_{\cB^N_\ell(\RR^6)}^2.
\end{align*}

It remains to evaluate
\begin{align*}
M_{122}&=(W_\ell\p^\alpha_\beta \Gamma(g_2,g_2), W_\ell\p^\alpha_\beta g_2)_{L^2(\RR^6)}.
\end{align*}
For this, we can apply Proposition \ref{prop4.2} to have
 \begin{align*}
 M_{122}&\lesssim
 \|g_2\|_{H^N_\ell(\RR^6)}||g_2||_{\cB^N_\ell(\RR^6)}^2.
\end{align*}
In conclusion, we have proved
\begin{align*}
M_1
\lesssim \cE_{N,\ell}^{1/2}\cD_{N,\ell}.
\end{align*}

By using integration by parts  and taking into account that $s+\gamma/2> 0$, we get
\begin{align*}
M_2&\lesssim |(W_\ell\p^\alpha_\beta[\pP, v \cdot\nabla_x]g,
W_\ell\p^\alpha_\beta g_2)_{L^2(\RR^6)}|
\\&\lesssim
(\|\p^\alpha \nabla_x\cA\|_{L^2(\RR^3)}+\|\p^\alpha \nabla_xg_2\|_{L^2(\RR^6 )})
\|\p^\alpha g_2\|_{L^2(\RR^6)}
\\&
\lesssim
\delta_0\|\nabla\cA\|_{H^{N-1}(\RR^3)}^2+ C_{\delta_0}
\|g_2\|_{H^N(\RR^3_x, L^2(\RR^3_v))}^2,
\end{align*}
where $\delta_0>0$ is a small constant.

Similarly, For $1\le|\beta|\le N$,
\begin{align*}
M_3\lesssim
\|W_\ell \p^{\alpha+1}_{\beta-1}g_2\|_{L^2(\RR^6)}
\ \|W_\ell \p^{\alpha}_{\beta}g_2\|_{L^2(\RR^6)}
\le C_\delta \|\p^{\alpha+1}_{\beta-1}g_2\|_{\cB^0_\ell(\RR^6)}^2
+\delta \|g_2\|_{\cB^N_\ell(\RR^6)}^2,
\end{align*}
where $\delta>0$ is another small constant.

The main ingredients of the estimation on $M_4$
are the commutator estimates  $I$ and $II$ which are
defined in  the proof of Proposition 4.8 in \cite{amuxy4-1}. Note that they
 are valid in general for $\gamma >-3$ so that the estimates there can be used here. That is,
\begin{align*}
|I|&=|([W_\ell,  \cL_1] g , W_\ell g)_{L^2(\RR^3)}|
\lesssim
\|W_\ell g \|_{L^2_{\gamma/2}(\RR^3)}^2.
\\
|II|&=  |(W_\ell [\partial_\beta, \cL_1 ]g, W_\ell \partial_\beta g)_{L^2(\RR^3)}
\\&
\lesssim \sum_{\beta_1 +\beta_2 +\beta_3=\beta, \ \beta_2 \neq 0}
| (W_\ell\cT(\p_{\beta_1}\mu^{1/2}, \partial_{\beta_2} g, \p_{\beta_3}\mu^{1/2}),
W_\ell \partial_\beta g)_{L^2(\RR^3)}|
\\&\lesssim
\Big( \sum_{\beta_1 +\beta_2 =\beta, \ \beta_2 \neq 0} ||| W_\ell \partial_{\beta_1} g|||_{\Phi_\gamma} \Big) ||| W_\ell \partial_\beta g |||_{\Phi_\gamma} .
\end{align*}
With this, later we  also need the following interpolation inequality
\begin{equation*}
\|W_\ell\p_\beta h\|_{L^2_{\gamma/2}}\lesssim
C_\delta\|\p_{\beta} h\|_{L^2_{s+\gamma/2}}+\delta
\|W_\ell\p_\beta h\|_{L^2_{s+\gamma/2}}
\lesssim
C_\delta\|| \p_{\beta}h\||_{\Phi_\gamma}+\delta
|||W_\ell\p_\beta h|||_{\Phi_\gamma}.
\end{equation*}
For the term $M_4$, using Proposition \ref{prop3.2.4}
\begin{align*}
M_4&\le |([W_\ell\p^\alpha_{\beta},\cL_1]g_2,W_\ell\p^\alpha_{\beta}g_2)_{L^2(\RR^6)}|
\\&\lesssim
\Big(\|g_2\|_{H^{|\alpha|+|\beta|}_\ell}+\sum_{|\alpha'|=|\alpha|, |\beta'|=|\beta|-1}\|\p^{\alpha'}_{\beta'}g_2\|_{\cB^{0}_\ell(\RR^3)}\Big)
\|g_2\|_{\cB^{|\alpha|+|\beta|}_\ell}
\end{align*}

Finally,  in view of Proposition \ref{prop4.4}, we have
\begin{align*}
M_5&=|(W_\ell\p^\alpha_{\beta}\cL_2g_2, W_\ell\p^\alpha_{\beta}g_2)_{L^2(\RR^6)}|
\le |(\p^\alpha_{\beta}\cL_2 g_2, W_{2\ell}\p^\alpha_{\beta}g_2)_{L^2(\RR^6)}|
\\&\le
C\| g_2\|_{H^{|\alpha|+|\beta|}(\RR^6)}
\|\mu^{1/10^3}W_{2\ell}\p^\alpha_{\beta}g_2\|_{L^2(\RR^6)}
\lesssim\| g_2\|_{H^{|\alpha|+|\beta|}(\RR^6)}^2.
\end{align*}
Now interpolation inequality
$$
\|g_2\|_{H^{|\alpha|+|\beta|}_\ell}+
\| g_2\|_{H^{|\alpha|+|\beta|}(\RR^6)}\leq \delta \|g_2\|_{\cB^{|\alpha|+|\beta|}_\ell}
+C_{\delta}\sum_{|\alpha'|=|\alpha|+1, |\beta'|=|\beta|-1}\|\p^{\alpha'}_{\beta'}g_2\|_{\cB^{0}_\ell}\,.
$$

 Plugging all these estimates to \eqref{energyweight} and fix the small constant
 $\delta$ with respect to the coefficient in front of the dissipation rate, we have
 for $|\beta|\neq 0$,
\begin{align*}
 \frac{d}{dt}&\| W_\ell\p^\alpha_\beta g_2\|^2_{L^2(\RR^6)}+
|| W_\ell\p^\alpha_\beta g_2 ||_{\cB^0_0}
\\&\lesssim
\|g_2\|_{H^N(\RR^3_x, L^2(\RR^3_v))}^2
+\cE_{N,\ell}^{1/2}\cD_{N,\ell}
\\&\notag+\delta_0\|\nabla_x\cA\|_{H^{N-1}(\RR^3)}^2
+\sum_{|\alpha'|=|\alpha|+1, |\beta'|=|\beta|-1}\|\p^{\alpha'}_{\beta'}g_2\|_{\cB^{0}_\ell}^2.
\end{align*}
By induction on $|\beta|$ and $|\alpha|+|\beta|$, we have
\begin{align}\label{energyalphabeta-2}
\frac{d}{dt}\cE_{N,\ell}+\cD_{N,\ell}\lesssim \|g_2\|_{H^N(\RR^3_x, L^2(\RR^3_v))}^2+\delta_0\|\nabla_x\cA\|_{H^{N-1}(\RR^3)}^2+\cE_{N,\ell}^{1/2}\cD_{N,\ell}.
\end{align}
Taking a suitable linear combination of
the estimates \eqref{pabc}, \eqref{microenergy11}, and \eqref{energyalphabeta-2}, we then
conclude Proposition \ref{energy}. And this completes the energy
estimate on the solution so that the global existence follows in the standard
way for small perturbation.

\bigskip

\noindent
{\bf Acknowledgements:}
The research of the first author was supported in part by Zhiyuan foundation and Shanghai Jiao Tong University. The research of the second author was
supported by  Grant-in-Aid for Scientific Research No.22540187,
Japan Society of the Promotion of Science.
The last author's research was supported by the General Research
Fund of Hong Kong,
 CityU No.103109, and the Lou Jia Shan Scholarship programme of
Wuhan University. The authors would like to
thank the financial supports from City University of Hong Kong, Kyoto
University, Rouen University and Wuhan University for their visits.

\end{document}